\documentclass{article}
\usepackage{amsfonts}
\usepackage{amsmath}
\usepackage{graphicx}

\newtheorem{theorem}{Theorem}

\newtheorem{lemma}[theorem]{Lemma}

\newtheorem{remark}[theorem]{Remark}

\newenvironment{proof}[1][Proof]{\noindent\textbf{#1.} }{\ \rule{0.5em}{0.5em}}
\numberwithin{equation}{section}
\def \dsum {\sum}
\def\QATOPD#1#2#3#4{{#3 \atopwithdelims#1#2 #4}}

\begin{document}

\title{Population processes sampled at random times}
\author{Luisa Beghin\thanks{%
e-mail: \texttt{luisa.beghin@uniroma1.it}} \and Enzo Orsingher\thanks{%
e-mail: enzo.orsingher@uniroma1.it}}
\date{}
\maketitle

\begin{abstract}
\noindent In this paper we study the iterated birth process of which we
examine the first-passage time distributions and the hitting probabilities.
Furthermore, linear birth processes, linear and sublinear death processes at
Poisson times are investigated. In particular, we study the hitting times in
all cases and examine their long-range behavior.

The time-changed population models considered here display upward (birth
process) and downward jumps (death processes) of arbitrary size and, for
this reason, can be adopted as adequate models in ecology, epidemics and
finance situations, under stress conditions.

\textbf{Keywords and phrases}: Yule-Furry process, linear and sublinear
death processes, hitting times, extinction probabilities, first-passage
times, Stirling numbers, Bell polynomials.

\noindent \emph{AMS Mathematical Subject Classification (2010): 60G55, 60J80.%
}
\end{abstract}

\section{Introduction}

We study here different compositions of point processes, say $K(t):=L(H(t)),$
$t\geq 0,$ where $L$ and $H$ are mutually independent and can possibly be
birth, death or homogeneous Poisson processes. These compositions arise in
the analysis of different population models, when, in some specific
experimental circumstances, the time $t$ must be replaced by a stochastic
process $H(t),$ $t\geq 0.$ In our analysis we concentrate our attention on
the case where the time increments are unit-valued and thus are properly
represented by point processes, such as linear birth and Poisson processes.
In these cases the composed process can be regarded as a randomly sampled
population model.

The simplest case of the composition of two independent homogeneous Poisson
processes (with different rates) has been already studied in \cite{OP}. More
recently, the iterated Poisson process has been considered in \cite{DIC}.

Our investigation concentrates here on the distribution of the first-passage
times, i.e.%
\begin{equation*}
T_{k}:=\inf \{t\geq 0:K(t)=k\},\quad k\geq 0,
\end{equation*}%
which, in this context, substantially differ from the upcrossing and
downcrossing times, i.e.%
\begin{equation*}
U_{k}:=\inf \{t\geq 0:K(t)\geq k\},\qquad V_{k}:=\inf \{t\geq 0:K(t)\leq
k\},\quad k\geq 0.
\end{equation*}%
Indeed, the composed processes $K$ jump over (resp. under) any level $k$
with positive probability. The hitting probabilities $\Pr (T_{k}<\infty )$
display different behaviors. For example, for the iterated birth process $%
K(t)=B_{1}(B_{2}(t))$ (with $B_{1},B_{2}$ independent linear birth
processes), the probabilities $\Pr (T_{k}<\infty )$ attain their maximal
value at $k=3$. In the linear death process at Poisson distributed time the
hitting probabilities display an oscillating behavior, which can be
perceived when the initial size of the population is sufficiently large.

Furthermore, we observe that in all the cases considered here the hitting
probabilities $\Pr (T_{k}<\infty )$ depend only on the parameter of the
outer process.

Moreover, the processes $K$ display sample paths with upward or downward
jumps of arbitrary size. Other models with the same feature, such as the
space-fractional Poisson processes or, in general, time-changed Poisson
processes with Bern\v{s}tein subordinators, have been analyzed in \cite{ORT}%
. Multiple jumps are also displayed by the generalized fractional birth
processes studied in \cite{ABR}. As in this last work, this property is
reflected in the form of the equation governing the state probabilities $%
p_{k}(t):=\Pr \{K(t)=k\},$ $k\geq 0,$ where the time derivative $%
p_{k}^{\prime }(t)$ is shown to depend, also here, on all $p_{k-j}(t)$, for $%
j=0,...,k.$

In particular, our analysis concerns the following specific cases:

\begin{itemize}
\item $B_{1}(B_{2}(t)),$ $t\geq 0$\qquad (section 3)

\item $B(N(t)),$ $t\geq 0$\qquad (section 4)

\item $D(N(t)),$ $t\geq 0$\qquad (section 5)

\item $\widetilde{D}(N(t)),$ $t\geq 0$\qquad (section 6),
\end{itemize}

\noindent where $B_{1}$ and $B_{2}$ are independent linear birth
(Yule-Furry) processes, $N$ is the homogeneous Poisson process, $D$ is the
linear death process and $\widetilde{D}$ is the sublinear death process. The
latter is characterized by the fact that the probability of a further death
in $[t,t+dt)$ depends on the number of deaths recorded up to time $t$
(unlike the linear case where this probability depends on the number of
surviving individuals). The sublinear birth process was introduced in \cite%
{DON}, while the sublinear death process has been investigated in \cite{OPS}%
, in a fractional context.

The processes considered here can be useful to model the population
evolution sampled at Poisson times. The iterated birth process can be more
appropriate for the cases where the time separating occurrences is rapidly
decreasing. This structure can be applied, for example, in some experimental
studies of diseases or epidemic diffusions.

Time-changed birth processes of different forms have been studied in \cite%
{DIN} and \cite{DIN2}, with various applications to finance. Time-changed
Poisson and birth processes arise also in the study of fractional point
processes (see, for example, \cite{KUM}, \cite{CAH1}, \cite{CAH3} and \cite%
{CAH2}).

\section{Notations}

\textbf{List of the main symbols}

\begin{equation*}
\begin{array}{lll}
& \text{PROCESS} & \text{STATE PROBAB.} \\
\text{Linear birth} & B_{\alpha }(t) &  \\
\text{Poisson} & N_{\lambda }(t) &  \\
\text{Linear death} & D_{\mu }(t) & p_{k}^{D}(t),\;0\leq k\leq n_{0} \\
\text{Sublinear death } & \widetilde{D}_{\mu }(t) & p_{k}^{\widetilde{D}%
}(t),\;0\leq k\leq n_{0} \\
\text{Iterated birth} & \mathcal{Z}(t)=B_{\alpha }(B_{\lambda }(t)) & q_{k}^{%
\mathcal{Z}}(t),\;k\geq 1 \\
\text{Linear birth at Poisson times} & \mathcal{X}(t)=B_{\alpha }(N_{\lambda
}(t)) & q_{k}^{\mathcal{X}}(t),\;k\geq 1 \\
\text{Linear death at Poisson times} & \mathcal{Y}(t)=D_{\mu }(N_{\lambda
}(t)) & q_{k}^{\mathcal{Y}}(t),\;0\leq k\leq n_{0} \\
\text{Sublinear death at Poisson times} & \widetilde{\mathcal{Y}}(t)=%
\widetilde{D}_{\mu }(N_{\lambda }(t)) & q_{k}^{\widetilde{Y}}(t),\;0\leq
k\leq n_{0}%
\end{array}%
\end{equation*}%
where $n_{0}\geq 1$ is the number of the initial components of the
population and $\lambda ,\mu ,\alpha >0.$

\section{The iterated linear birth process}

Let us first consider the composition of a linear birth (Yule-Furry) process
$B_{\alpha }:=B_{\alpha }(t),t\geq 0$, with one initial progenitor and birth
rate $\alpha >0,$ with an independent process of the same kind, with
parameter $\lambda >0,$ i.e.%
\begin{equation}
\mathcal{Z}(t):=B_{\alpha }(B_{\lambda }(t)),\qquad t\geq 0.  \label{def}
\end{equation}%
As a consequence of the definition (\ref{def}), the initial number of
progenitors of the iterated linear birth process is random, since $\mathcal{Z%
}(0)=B_{\alpha }(1)$ a.s.

The probability mass function of this process can be written, for any $k\geq
1,$ as
\begin{eqnarray*}
q_{k}^{\mathcal{Z}}(t)&:=&\Pr \{\mathcal{Z}(t)=k\}=\mathbb{E}\Pr
\{B_{\alpha }(B_{\lambda }(t))=k|B_{\lambda }(t)\} \\
&=&\dsum\limits_{j=1}^{\infty }\Pr \{B_{\alpha }(j)=k\}\Pr \{B_{\lambda
}(t)=j\} \\
&=&e^{-\lambda t}\dsum\limits_{j=1}^{\infty }e^{-\alpha j}\left(
1-e^{-\alpha j}\right) ^{k-1}\left( 1-e^{-\lambda t}\right) ^{j-1} \\
&=&e^{-\lambda t-\alpha }\sum_{l=0}^{k-1}\binom{k-1}{l}\frac{(-e^{-\alpha
})^{l}}{1-e^{-\alpha (l+1)}(1-e^{-\lambda t})}.
\end{eqnarray*}%
We observe that%
\begin{eqnarray*}
q_{1}^{\mathcal{Z}}(t) &=&\frac{e^{-\lambda t-\alpha }}{1-e^{-\alpha
}(1-e^{-\lambda t})}=\mathbb{E}e^{-\alpha B_{\lambda }(t)} \\
q_{2}^{\mathcal{Z}}(t) &=&q_{1}^{\mathcal{Z}}(t)\frac{1-e^{-\alpha }}{%
1-e^{-2\alpha }(1-e^{-\lambda t})}\leq q_{1}^{\mathcal{Z}}(t).
\end{eqnarray*}%
We prove now that the probabilities $q_{k}^{\mathcal{Z}}(t),$ $k\geq 1,$ $%
t\geq 0$ decrease as $k$ increases, since we have that%
\begin{eqnarray*}
q_{k}^{\mathcal{Z}}(t)-q_{k-1}^{\mathcal{Z}}(t) &=&e^{-\lambda t-\alpha }%
\left[ \sum_{l=1}^{k-2}\binom{k-2}{l-1}\frac{(-e^{-\alpha })^{l}}{%
1-e^{-\alpha (l+1)}(1-e^{-\lambda t})}+\frac{(-e^{-\alpha })^{k-1}}{%
1-e^{-\alpha k}(1-e^{-\lambda t})}\right]  \\
&=&e^{-\lambda t-\alpha }\sum_{l=1}^{k-1}\binom{k-2}{l-1}\frac{(-e^{-\alpha
})^{l}}{1-e^{-\alpha (l+1)}(1-e^{-\lambda t})} \\
&=&-e^{-\lambda t-2\alpha }\sum_{m=0}^{\infty }(1-e^{-\lambda
t})^{m}e^{-2\alpha m}\sum_{l=0}^{k-2}\binom{k-2}{l}(-1)^{l}e^{-\alpha l(m+1)}
\\
&=&-e^{-\lambda t-2\alpha }\sum_{m=0}^{\infty }(1-e^{-\lambda
t})^{m}e^{-2\alpha m}(1-e^{-\alpha (m+1)})^{k-2}<0,
\end{eqnarray*}%
for any $k\geq 2.$

The sample paths of the iterated birth process display upward jumps of size
larger or equal to one. For this reason the analysis of the first-passage
time through an arbitrary level $k\geq 2$, i.e.%
\begin{equation*}
T_{k}^{\mathcal{Z}}=\inf \{t>0:\mathcal{Z}(t)=k\},
\end{equation*}%
is of a certain importance. Moreover, we shall prove that $\Pr \{T_{k}^{%
\mathcal{Z}}<\infty \}$ is strictly less than one, as it happens for the
iterated Poisson process (see \cite{OPS}) and thus any level $k$ can be
avoided with positive probability, because of "multiple jumps". We present
the explicit distribution of $T_{k}^{\mathcal{Z}}$ in the next theorem.

\begin{theorem}
The distribution of $T_{k}^{\mathcal{Z}}$ reads%
\begin{equation}
\Pr \{T_{k}^{\mathcal{Z}}\in dt\}/dt=\lambda e^{-\lambda t-\alpha
k}\sum_{j=1}^{\infty }je^{-\alpha j}(1-e^{-\lambda t})^{j-1}\left[
(e^{\alpha }-e^{-\alpha j})^{k-1}-(1-e^{-\alpha j})^{k-1}\right] ,
\label{bb}
\end{equation}%
for any $k\geq 2,$ $0\leq t<\infty .$
\end{theorem}

\begin{proof}
By a conditioning argument we get \
\begin{eqnarray}
\Pr \{T_{k}^{\mathcal{Z}}\in dt\}&=&\mathbb{E}\Pr \{T_{k}^{\mathcal{Z}}\in
dt|B_{\alpha }(1)\}  \notag \\
&=&\sum_{h=1}^{k-1}\Pr \{B_{\alpha }(B_{\lambda }(t))=k-h,B_{\alpha
}(B_{\lambda }(t+dt))=k\}  \notag \\
&=&\lambda dt\sum_{j=1}^{\infty }j\sum_{h=1}^{k-1}\Pr \{\left. B_{\alpha
}(B_{\lambda }(t))=k-h,B_{\alpha }(B_{\lambda }(t)+1)=k\right\vert
B_{\lambda }(t)=j\}\Pr \{B_{\lambda }(t)=j\}  \notag \\
&=&\lambda dt\sum_{j=1}^{\infty }j\Pr \{B_{\lambda
}(t)=j\}\sum_{h=1}^{k-1}\Pr \{\left. B_{\alpha }(j+1)=k\right\vert B_{\alpha
}(j)=k-h\}\Pr \{B_{\alpha }(j)=k-h\}  \notag \\
&=&\lambda e^{-\lambda t-\alpha k}dt\sum_{j=1}^{\infty }je^{-\alpha
j}(1-e^{-\lambda t})^{j-1}\sum_{h=1}^{k-1}\binom{k-1}{h}(e^{\alpha
}-1)^{h}(1-e^{-\alpha j})^{k-h-1},  \notag
\end{eqnarray}%
which coincides with (\ref{bb}).
\end{proof}

We note that, in the special case $k=2$, we get%
\begin{equation*}
\Pr \{T_{2}^{\mathcal{Z}}\in dt\}=\frac{\lambda e^{-\lambda t-2\alpha
}(1-e^{-\alpha })}{\left[ 1-e^{-\alpha }(1-e^{-\lambda t})\right] ^{2}}%
dt=\lambda e^{\lambda t}(1-e^{-\alpha })[q_{1}^{\mathcal{Z}}(t)]^{2}dt
\end{equation*}

\begin{remark}
By integrating (\ref{bb}), we get that%
\begin{equation}
\Pr \{T_{k}^{\mathcal{Z}}<\infty \}=e^{-\alpha k}\sum_{j=1}^{\infty
}e^{-\alpha j}\left[ (e^{\alpha }-e^{-\alpha j})^{k-1}-(1-e^{-\alpha
j})^{k-1}\right] ,  \label{cc}
\end{equation}%
which can be alternatively written as%
\begin{eqnarray}
\Pr \{T_{k}^{\mathcal{Z}}<\infty \} &=&\sum_{j=2}^{\infty }e^{-\alpha
j}(1-e^{-\alpha j})^{k-1}-e^{-\alpha k}\sum_{j=1}^{\infty }e^{-\alpha
j}(1-e^{-\alpha j})^{k-1}  \notag \\
&=&(1-e^{-\alpha k})\sum_{j=2}^{\infty }e^{-\alpha j}(1-e^{-\alpha
j})^{k-1}-e^{-\alpha k-\alpha }(1-e^{-\alpha })^{k-1}  \notag \\
&=&(1-e^{-\alpha k})\sum_{j=2}^{\infty }\Pr \{B_{\alpha }(j)=k\}-e^{-\alpha
k}\Pr \{B_{\alpha }(1)=k\}.  \notag \\
&&  \label{pre}
\end{eqnarray}%
In order to analyze some special cases, we supply also the following finite
sum form of the hitting probabilities given in (\ref{cc}):%
\begin{eqnarray}
\Pr \{T_{k}^{\mathcal{Z}}<\infty \} &=&e^{-\alpha k}\sum_{j=1}^{\infty
}e^{-\alpha j}\sum_{r=0}^{k-1}\binom{k-1}{r}\left[ (-e^{-\alpha
j})^{r}e^{\alpha (k-r-1)}-(-e^{-\alpha j})^{r}\right]  \notag \\
&=&e^{-\alpha k}\sum_{r=0}^{k-2}\binom{k-1}{r}(-1)^{r}\frac{e^{-\alpha (r+1)}%
}{1-e^{-\alpha (r+1)}}\left( e^{\alpha (k-r-1)}-1\right) .  \label{ic}
\end{eqnarray}%
It is now easy to show that the following relationships hold, for $k=2,3,4$:%
\begin{equation*}
\Pr \{T_{4}^{\mathcal{Z}}<\infty \}<\Pr \{T_{2}^{\mathcal{Z}}<\infty \}<\Pr
\{T_{3}^{\mathcal{Z}}<\infty \},
\end{equation*}%
since we have that%
\begin{eqnarray*}
\Pr \{T_{2}^{\mathcal{Z}}<\infty \} &=&e^{-2\alpha } \\
\Pr \{T_{3}^{\mathcal{Z}}<\infty \} &=&e^{-2\alpha }\left[ 1+e^{-\alpha }%
\frac{1-e^{-\alpha }}{1+e^{-\alpha }}\right] \\
\Pr \{T_{4}^{\mathcal{Z}}<\infty \} &=&e^{-2\alpha }\left[ 1-e^{-\alpha }%
\frac{(1-e^{-\alpha })(1+e^{-3\alpha })}{1+e^{-\alpha }+e^{-2\alpha }}\right]
.
\end{eqnarray*}%
The following figures (produced by the software R) describe the behavior of
the probabilities (\ref{ic}), for different values of $\alpha .$ We remark
that the scales are not the same in different figures.
\end{remark}

\begin{figure}[htp!]
\caption{Hitting times $\Pr \{T_{k}^{\mathcal{Z}}<\infty \}$}
\label{label}\centering
\par
\includegraphics[scale=0.30]{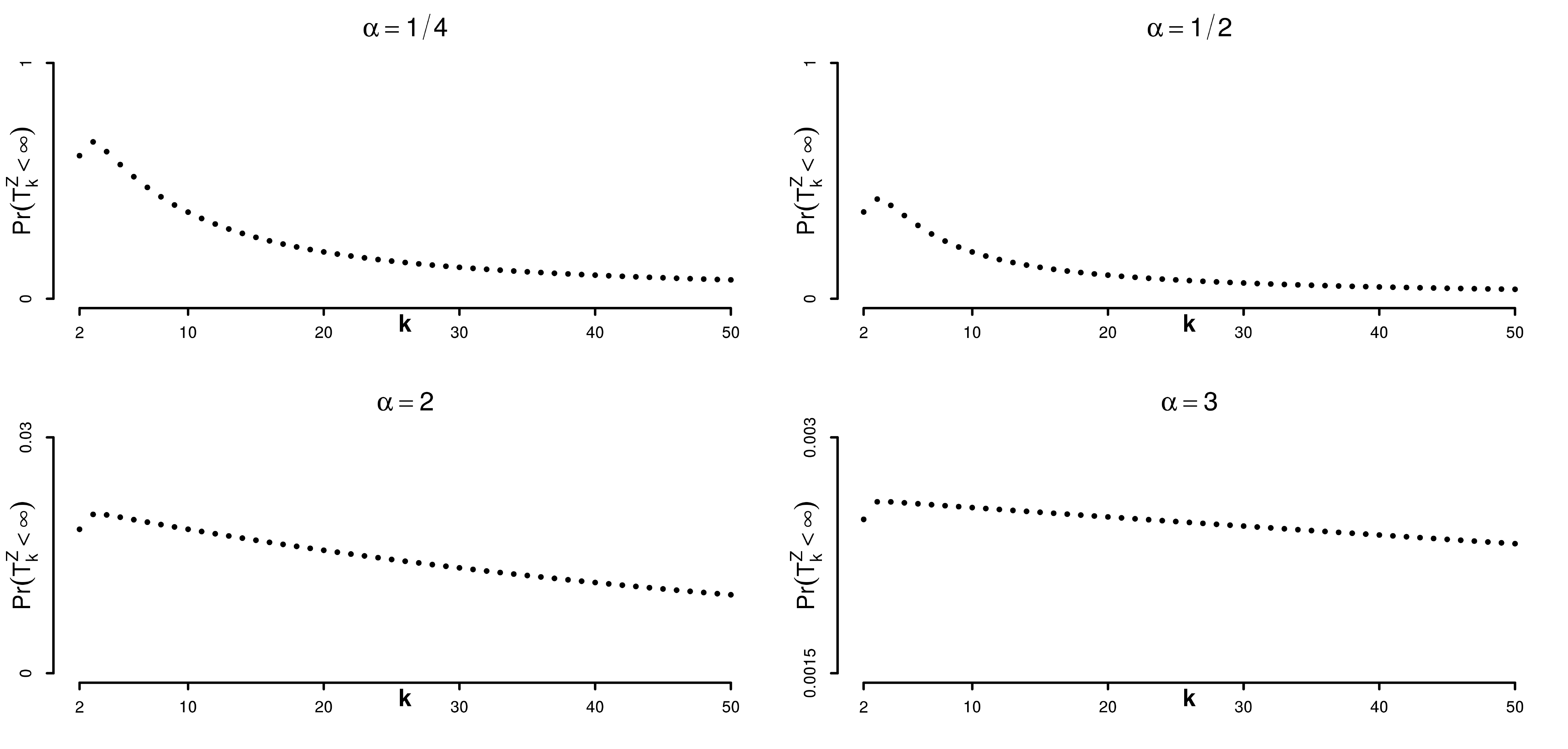}
\end{figure}

The figures above show that the probabilities $\Pr \{T_{k}^{\mathcal{Z}%
}<\infty \}$, for $k>3$, are decreasing functions which, for $0<\alpha <1,$
vanish more rapidly than for $\alpha >1.$

The rate of decreasing is slow, for all $\alpha ,$ since, from (\ref{pre}),
we get that%
\begin{eqnarray*}
\sum_{k=2}^{\infty }\Pr \{T_{k}^{\mathcal{Z}}<\infty \}
&=&\sum_{j=2}^{\infty }e^{-\alpha j}\sum_{k=2}^{\infty }(1-e^{-\alpha
j})^{k-1}-\sum_{j=1}^{\infty }e^{-\alpha (j+1)}\sum_{k=2}^{\infty
}(e^{-\alpha }-e^{-\alpha (j+1)})^{k-1} \\
&=&\sum_{j=2}^{\infty }(1-e^{-\alpha j})-\sum_{j=1}^{\infty }e^{-\alpha
(j+1)-\alpha }\frac{1-e^{-\alpha j}}{1-e^{-\alpha }+e^{-\alpha (j+1)}} \\
&=&\sum_{j=2}^{\infty }\frac{(1-e^{-\alpha j})[1-e^{-\alpha }-e^{-\alpha
(j+1)}-e^{-\alpha (j+1)-\alpha }]}{1-e^{-\alpha }+e^{-\alpha (j+1)}}-\frac{%
e^{-3\alpha }(1-e^{-\alpha })}{1-e^{-\alpha }+e^{-2\alpha }} \\
&=&(1-e^{-\alpha })\left\{ \sum_{j=2}^{\infty }\frac{(1-e^{-\alpha
j})[1+e^{-\alpha (j+1)}]}{1-e^{-\alpha }+e^{-\alpha (j+1)}}-\frac{%
e^{-3\alpha }}{1-e^{-\alpha }+e^{-2\alpha }}\right\} =\infty ,
\end{eqnarray*}%
since%
\begin{equation*}
\lim_{j\rightarrow \infty }\frac{(1-e^{-\alpha j})[1+e^{-\alpha (j+1)}]}{%
1-e^{-\alpha }+e^{-\alpha (j+1)}}=\frac{1}{1-e^{-\alpha }}.
\end{equation*}

\section{Linear birth process at Poisson times}

We now pass to the analysis of the process defined as
\begin{equation*}
\mathcal{X}(t):=B_{\alpha }(N_{\lambda }(t)),\qquad t\geq 0,
\end{equation*}%
obtained as a composition of a linear birth process $B_{\alpha }$ with an
independent homogeneous Poisson process $N_{\lambda }(t),t\geq 0,$ with rate
$\lambda >0.$ The process $\mathcal{X}$ can be regarded as a linear birth
process randomly sampled at Poisson times. We first note that, for $k\geq
n_{0},$\
\begin{eqnarray}
q_{k,n_{0}}^{\mathcal{X}}(t)&:=&\Pr \{\left. \mathcal{X}(t)=k\right\vert
\mathcal{X}(0)=n_{0}\}  \label{si} \\
&=&e^{-\lambda t}\dsum\limits_{j=0}^{\infty }\binom{k-1}{k-n_{0}}e^{-\alpha
jn_{0}}\left( 1-e^{-\alpha j}\right) ^{k-n_{0}}\frac{(\lambda t)^{j}}{j!}
\notag \\
&=&\sum_{l=0}^{k-n_{0}}\binom{k-n_{0}}{l}(-1)^{l}\exp \left\{ \lambda
t[1-e^{-\alpha (n_{0}+l)}]\right\} .  \notag
\end{eqnarray}%
In particular, for $k=n_{0}$, we get%
\begin{equation}
q_{n_{0},n_{0}}^{\mathcal{X}}(t)=e^{-\lambda t(1-e^{-\alpha n_{0}})}
\label{st}
\end{equation}%
and thus the waiting time\ for the appearance of the first offspring is
exponentially distributed with parameter $\lambda (1-e^{-\alpha n_{0}}).$ It
is easy to show that
\begin{equation}
\mathbb{E}\mathcal{X}(t)=n_{0}e^{\lambda t(e^{\alpha }-1)}.  \label{bu}
\end{equation}%
The variance of $\mathcal{X}$ can be obtained as follows
\begin{eqnarray}
Var\mathcal{X}(t) &=&Var\mathbb{E}\left[ \left. B_{\alpha }(N_{\lambda
}(t))\right\vert N_{\lambda }(t)\right] +\mathbb{E}Var\left[ \left.
B_{\alpha }(N_{\lambda }(t))\right\vert N_{\lambda }(t)\right]  \label{he} \\
&=&Var\left[ n_{0}e^{\alpha N_{\lambda }(t)}\right] +\mathbb{E}\left[
n_{0}e^{2\alpha N_{\lambda }(t)}-n_{0}e^{\alpha N_{\lambda }(t)}\right]
\notag \\
&=&n_{0}(n_{0}+1)e^{\lambda t(e^{2\alpha }-1)}-n_{0}^{2}e^{\lambda
t(e^{\alpha }-1)}-n_{0}e^{2\lambda t(e^{\alpha }-1)}.  \notag
\end{eqnarray}%
For the sake of simplicity, from now on, we assume $n_{0}=1$ and denote for
brevity $q_{k,1}^{\mathcal{X}}$ as $q_{k}^{\mathcal{X}}$. The factorial
moments of $\mathcal{X}$ can be evaluated as follows:
\begin{eqnarray}
&&\mathbb{E}\left\{ \left. \mathcal{X}(t)\left[ \mathcal{X}(t)-1\right] ...[%
\mathcal{X}(t)-r+1]\right\vert \mathcal{X}(t)=1\right\}  \label{ci} \\
&=&\sum_{k=r}^{\infty }k(k-1)...(k-r+1)q_{k}^{\mathcal{X}}(t)  \notag \\
&=&[\text{by (\ref{si})]}  \notag \\
&=&r!e^{-\lambda t}\sum_{k=r}^{\infty }\binom{k}{r}\sum_{j=0}^{\infty
}e^{-\alpha j}\left( 1-e^{-\alpha j}\right) ^{k-1}\frac{(\lambda t)^{j}}{j!}
\notag \\
&=&r!e^{-\lambda t}\sum_{j=0}^{\infty }e^{-\alpha j}\frac{(\lambda t)^{j}}{j!%
}\sum_{l=0}^{\infty }\binom{r+l}{l}\left( 1-e^{-\alpha j}\right) ^{r+l-1}
\notag \\
&=&r!e^{-\lambda t}\sum_{j=0}^{\infty }e^{-\alpha j}\frac{(\lambda t)^{j}}{j!%
}\frac{\left( 1-e^{-\alpha j}\right) ^{r-1}}{[1-(1-e^{-\alpha j})]^{r+1}}
\notag
\end{eqnarray}
\begin{eqnarray}
&=&r!e^{-\lambda t}\sum_{j=0}^{\infty }\frac{(\lambda te^{\alpha r})^{j}}{j!}%
\left( 1-e^{-\alpha j}\right) ^{r-1}  \notag \\
&=&r!e^{-\lambda t}\sum_{j=0}^{\infty }\frac{(\lambda te^{\alpha r})^{j}}{j!}%
\sum_{m=0}^{r-1}\binom{r-1}{m}(-1)^{m}e^{-\alpha jm}  \notag \\
&=&r!\sum_{m=0}^{r-1}\binom{r-1}{m}(-1)^{m}\exp \left\{ -\lambda
t(1-e^{\alpha (r-m)})\right\} .  \notag
\end{eqnarray}%
By using result (\ref{ci}) we can check formulas (\ref{bu}) and (\ref{he}),
by means of independent calculations.

Our main interest here is to study the distribution of the first-passage
time for an arbitrary level $k$, i.e.%
\begin{equation*}
T_{k}^{\mathcal{X}}=\inf \{t>0:\mathcal{X}(t)=k\},
\end{equation*}%
for $k\geq 2.$ In the next theorem we derive the following explicit
distribution of $T_{k}^{\mathcal{X}}.$

\begin{theorem}
For $t\geq 0$ and $k\geq 2$ we have that
\begin{equation}
\Pr \{\left. T_{k}^{\mathcal{X}}\in dt\right\vert \mathcal{X}%
(0)=1\}/dt=\lambda e^{-\alpha k}\sum_{l=0}^{k-2}(-1)^{l}\binom{k-1}{l}%
(e^{\alpha (k-1-l)}-1)\exp \{-\lambda t[1-e^{-\alpha (l+1)}]\}  \label{two}
\end{equation}
\end{theorem}

\begin{proof}
We start by considering that%
\begin{eqnarray}
&&\Pr \{\left. T_{k}^{\mathcal{X}}\in dt\right\vert \mathcal{X}(0)=1\}
\notag \\
&=&\sum_{h=1}^{k-1}\Pr \{B_{\alpha }(N_{\lambda }(t))=k-h,B_{\alpha
}(N_{\lambda }(t+dt))=k\}  \notag \\
&=&\sum_{h=1}^{k-1}\Pr \{B_{\alpha }(N_{\lambda }(t))=k-h,B_{\alpha
}(N_{\lambda }(t)+dN_{\lambda }(t))=k\}  \notag \\
&=&\lambda dt\sum_{j=0}^{\infty }\sum_{h=1}^{k-1}\Pr \{\left. B_{\alpha
}(N_{\lambda }(t))=k-h,B_{\alpha }(N_{\lambda }(t)+1)=k\right\vert
N_{\lambda }(t)=j\}\Pr \{N_{\lambda }(t)=j\}  \notag \\
&=&\lambda e^{-\lambda t}dt\sum_{j=0}^{\infty }\frac{(\lambda t)^{j}}{j!}%
\sum_{h=1}^{k-1}\Pr \{B_{\alpha }(j)=k-h,B_{\alpha }(j+1)=k\}  \notag \\
&=&\lambda e^{-\lambda t}dt\sum_{j=0}^{\infty }\frac{(\lambda t)^{j}}{j!}%
\sum_{h=1}^{k-1}e^{-\alpha j}(1-e^{-\alpha j})^{k-h-1}\binom{k-1}{h}%
e^{-\alpha (k-h)}(1-e^{-\alpha })^{h}  \notag \\
&=&\lambda e^{-\lambda t-\alpha k}dt\sum_{h=1}^{k-1}\binom{k-1}{h}(e^{\alpha
}-1)^{h}\sum_{l=0}^{k-h-1}\binom{k-h-1}{l}(-1)^{l}\sum_{j=0}^{\infty }\frac{%
(\lambda te^{-\alpha (l+1)})^{j}}{j!}  \notag \\
&=&\lambda e^{-\lambda t-\alpha k}dt\sum_{h=1}^{k-1}\binom{k-1}{h}(e^{\alpha
}-1)^{h}\sum_{l=0}^{k-h-1}\binom{k-h-1}{l}(-1)^{l}\exp \{\lambda te^{-\alpha
(l+1)}\}  \notag \\
&=&\lambda e^{-\lambda t-\alpha k}dt\sum_{l=0}^{k-2}(-1)^{l}\exp \{\lambda
te^{-\alpha (l+1)}\}\sum_{h=1}^{k-l-1}\binom{k-1}{h}\binom{k-h-1}{l}%
(e^{\alpha }-1)^{h}  \notag \\
&=&\lambda e^{-\alpha k}dt\sum_{l=0}^{k-2}\binom{k-1}{l}(-1)^{l}\exp
\{-\lambda t(1-e^{-\alpha (l+1)})\}\sum_{h=1}^{k-l-1}\binom{k-l-1}{h}%
(e^{\alpha }-1)^{h},  \notag
\end{eqnarray}%
which coincides with (\ref{two}).
\end{proof}

\begin{remark}
In the particular case $k=2,$ formula (\ref{two}) reduces to%
\begin{eqnarray}
\Pr \{\left. T_{2}^{\mathcal{X}}\in dt\right\vert \mathcal{X}(0)=1\}/dt
&=&\lambda e^{-\alpha }(1-e^{-\alpha })e^{-\lambda t(1-e^{-\alpha })}  \notag
\\
&=&e^{-\alpha }\Pr \{Z\in dt\},  \notag
\end{eqnarray}%
where $Z$ is an exponential r.v. with parameter $\lambda (1-e^{-\alpha })$
and this clearly shows that for the first-passage time through the level $%
k=2,$ the following result holds%
\begin{equation}
\Pr \{\left. T_{2}^{\mathcal{X}}<\infty \right\vert \mathcal{X}%
(0)=1\}=e^{-\alpha }<1.  \label{rel}
\end{equation}
\end{remark}

\begin{remark}
For any $k\geq 2$, by integrating formula (\ref{two}), we get%
\begin{eqnarray}
&&\Pr \{\left. T_{k}^{\mathcal{X}}<\infty \right\vert \mathcal{X}(0)=1\}
\label{for} \\
&=&e^{-\alpha k}\sum_{l=0}^{k-2}(-1)^{l}\binom{k-1}{l}\frac{e^{\alpha
(k-1-l)}-1}{1-e^{-\alpha (l+1)}}  \notag \\
&=&e^{-\alpha k}\sum_{l=0}^{k-2}(-1)^{l}\binom{k-1}{l}[e^{\alpha
(k-1-l)}-1]\sum_{m=0}^{\infty }e^{-\alpha (l+1)m}  \notag \\
&=&e^{-\alpha k}\sum_{m=0}^{\infty }e^{-\alpha m}\sum_{l=0}^{k-2}(-1)^{l}%
\binom{k-1}{l}[e^{\alpha (k-1)-\alpha l(m+1)}-e^{-\alpha lm}]  \notag \\
&=&\sum_{m=0}^{\infty }e^{-\alpha (m+1)}\left[ (1-e^{-\alpha
(m+1)})^{k-1}-(-1)^{k-1}e^{-\alpha (m+1)(k-1)}\right] +  \notag \\
&&-e^{-\alpha k}\sum_{m=0}^{\infty }e^{-\alpha m}\left[ (1-e^{-\alpha
m})^{k-1}-(-1)^{k-1}e^{-\alpha m(k-1)}\right]   \notag \\
&=&\sum_{m=0}^{\infty }\left[ e^{-\alpha (m+1)}(1-e^{-\alpha
(m+1)})^{k-1}-e^{-\alpha k}e^{-\alpha m}(1-e^{-\alpha m})^{k-1}\right]
\notag \\
&=&(1-e^{-\alpha k})\sum_{m=1}^{\infty }e^{-\alpha m}(1-e^{-\alpha m})^{k-1}
\notag \\
&=&(1-e^{-\alpha k})\sum_{m=1}^{\infty }\Pr \{B_{\alpha }(m)=k\}.  \notag
\end{eqnarray}%
As in the case of the iterated birth process, the hitting probabilities (\ref%
{for}) are not affected by the parameter $\lambda $ of the inner process ($%
N_{\lambda }$, in this case)$.$ Moreover, we remark that the relationship
given in the last line is equal to the analogous one presented in (\ref{pre}%
) (except for the additional term in (\ref{pre}), which is due to the
different starting point of the birth process with respect to the Poisson
one):%
\begin{equation*}
\Pr \{\left. T_{k}^{\mathcal{X}}<\infty \right\vert \mathcal{X}(0)=1\}-\Pr
\{T_{k}^{\mathcal{Z}}<\infty \}=\Pr \{B_{\alpha }(1)=k\},
\end{equation*}%
for any $k\geq 2.$ Thus, for the iterated linear birth process, the
probability of reaching any level $k$ in a finite time is strictly smaller
than the corresponding probability for $B_{\alpha }(N_{\lambda }(t)),t>0$.
This difference decreases monotonically with $k.$
\end{remark}

\begin{remark}
For $k=2$ the first line of formula (\ref{for}) reduces to (\ref{rel})$.$ In
the special case $k=3$ we have that%
\begin{eqnarray*}
&&\Pr \{\left. T_{3}^{\mathcal{X}}<\infty \right\vert \mathcal{X}(0)=1\} \\
&=&e^{-3\alpha }\frac{e^{\alpha }-1}{1-e^{-\alpha }}\left[ e^{\alpha }+1-%
\frac{2}{1+e^{-\alpha }}\right] \\
&=&e^{-\alpha }\frac{1+e^{-2\alpha }}{1+e^{-\alpha }} \\
&=&\Pr \{\left. T_{2}^{\mathcal{X}}<\infty \right\vert \mathcal{X}(0)=1\}%
\frac{1+e^{-2\alpha }}{1+e^{-\alpha }}<\Pr \{\left. T_{2}^{\mathcal{X}%
}<\infty \right\vert \mathcal{X}(0)=1\}<1.
\end{eqnarray*}%
We show now that, for any $k\geq 2$ the distribution of the first-passage
time for the level $k$ is monotonically decreasing: we present the following
heuristic argument
\begin{eqnarray*}
&&\Pr \{\left. T_{k}^{\mathcal{X}}<\infty \right\vert \mathcal{X}(0)=1\}-\Pr
\{\left. T_{k-1}^{\mathcal{X}}<\infty \right\vert \mathcal{X}(0)=1\} \\
&=&(1-e^{-\alpha k})\sum_{m=0}^{\infty }e^{-\alpha m}(1-e^{-\alpha
m})^{k-1}-(1-e^{-\alpha (k-1)})\sum_{m=0}^{\infty }e^{-\alpha
m}(1-e^{-\alpha m})^{k-2} \\
&\simeq &(1-e^{-\alpha k})\int_{0}^{+\infty }e^{-\alpha x}(1-e^{-\alpha
x})^{k-1}dx-(1-e^{-\alpha (k-1)})\int_{0}^{+\infty }e^{-\alpha
x}(1-e^{-\alpha x})^{k-2}dx \\
&=&\frac{1-e^{-\alpha k}}{\alpha k}-\frac{1-e^{-\alpha (k-1)}}{\alpha (k-1)}
\\
&=&\frac{1-e^{-\alpha }}{\alpha }\left[ \frac{1+...+e^{-\alpha (k-1)}}{k}-%
\frac{1+...+e^{-\alpha (k-2)}}{k-1}\right] \\
&=&\frac{(1-e^{-\alpha })}{\alpha }\left\{ \left( 1+...+e^{-\alpha
(k-2)}\right) \left[ \frac{1}{k}-\frac{1}{k-1}\right] +\frac{e^{-\alpha
(k-1)}}{k}\right\} \\
&\leq &\frac{(1-e^{-\alpha })}{\alpha }\left[ -\frac{e^{-\alpha (k-2)}}{%
k(k-1)}(k-1)+\frac{e^{-\alpha (k-1)}}{k}\right] <0,
\end{eqnarray*}%
for $k>2.$
\end{remark}

The plots of Fig.2 confirm the decreasing structure of the probabilities $%
\Pr \{\left. T_{k}^{\mathcal{X}}<\infty \right\vert \mathcal{X}(0)=1\}.$

\begin{figure}[htp!]
\caption{Hitting times $\Pr \{T_{k}^{\mathcal{X}}<\infty \}$}
\label{label2}\centering
\par
\includegraphics[scale=0.30]{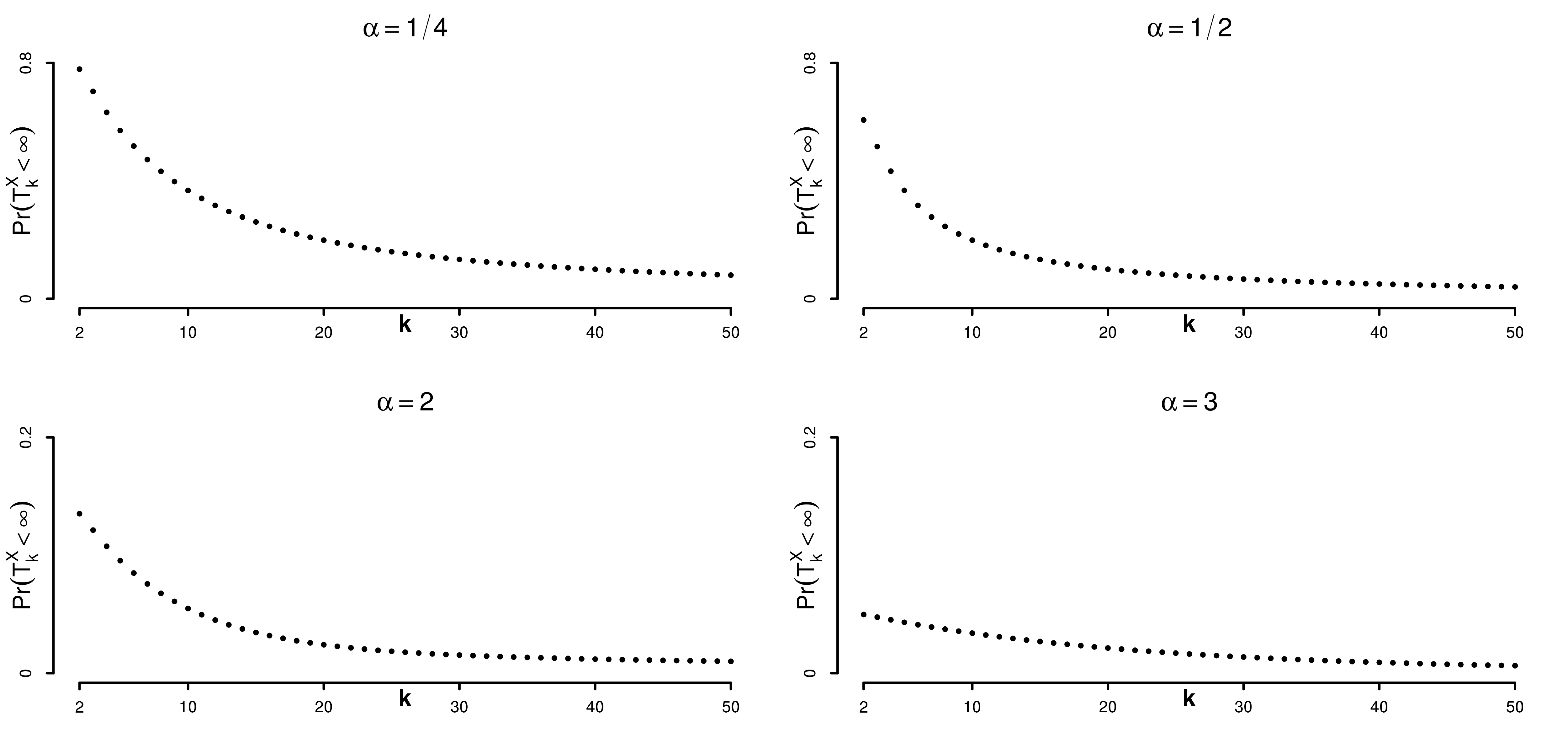}
\end{figure}

\begin{remark}
From formula (\ref{for}) we derive the following equality%
\begin{eqnarray}
g_{k}^{\mathcal{X}}-g_{k-1}^{\mathcal{X}}&:=&\frac{\Pr \{\left. T_{k}^{%
\mathcal{X}}<\infty \right\vert \mathcal{X}(0)=1\}}{1-e^{-\alpha k}}-\frac{%
\Pr \{\left. T_{k-1}^{\mathcal{X}}<\infty \right\vert \mathcal{X}(0)=1\}}{%
1-e^{-\alpha (k-1)}}  \notag \\
&=&-\sum_{m=1}^{\infty }e^{-2\alpha m}(1-e^{-\alpha m})^{k-2}<0.  \label{io}
\end{eqnarray}%
It is clear from (\ref{io}) that, for large values of $k$, we get%
\begin{equation*}
\Pr \{\left. T_{k}^{\mathcal{X}}<\infty \right\vert \mathcal{X}%
(0)=1\}\approx \Pr \{\left. T_{k-1}^{\mathcal{X}}<\infty \right\vert
\mathcal{X}(0)=1\}
\end{equation*}%
and thus the hitting probabilities slowly change with $k$ as Fig. 2 confirms.

By applying formula (\ref{io}) we can prove also the following relationship%
\begin{eqnarray}
&&\frac{\Pr \{\left. T_{k}^{\mathcal{X}}<\infty \right\vert \mathcal{X}%
(0)=1\}}{1-e^{-\alpha k}}-\frac{\Pr \{\left. T_{k-2}^{\mathcal{X}}<\infty
\right\vert \mathcal{X}(0)=1\}}{1-e^{-\alpha (k-2)}}  \label{in} \\
&=&-\sum_{m=1}^{\infty }e^{-2\alpha m}(1-e^{-\alpha m})^{k-3}\left(
2-e^{-\alpha m}\right)  \notag \\
&=&2\left[ \frac{\Pr \{\left. T_{k-1}^{\mathcal{X}}<\infty \right\vert
\mathcal{X}(0)=1\}}{1-e^{-\alpha (k-1)}}-\frac{\Pr \{\left. T_{k-2}^{%
\mathcal{X}}<\infty \right\vert \mathcal{X}(0)=1\}}{1-e^{-\alpha (k-2)}}%
\right] +\sum_{m=1}^{\infty }e^{-3\alpha m}(1-e^{-\alpha m})^{k-3}.  \notag
\end{eqnarray}%
Formula (\ref{in}) shows that, for large values of $k$,
\begin{equation*}
g_{k}^{\mathcal{X}}-g_{k-2}^{\mathcal{X}}\simeq 2(g_{k-1}^{\mathcal{X}%
}-g_{k-2}^{\mathcal{X}}).
\end{equation*}
\end{remark}

\section{Death processes at Poisson times}

We here consider linear and sublinear death processes at Poisson times. The
linear death process, with initial population of $n_{0}$ individuals and
death rate $\mu >0$, has a binomial distribution, i.e.%
\begin{equation}
p_{k}^{D}:=\Pr \{\left. D_{\mu }(t)=k\right\vert D_{\mu }(0)=n_{0}\}=\binom{%
n_{0}}{k}e^{-\mu tk}(1-e^{-\mu t})^{n_{0}-k},\qquad 0\leq k\leq n_{0}.
\label{i}
\end{equation}%
The probabilities (\ref{i}) satisfy the difference-differential equations%
\begin{equation}
\left\{
\begin{array}{l}
\frac{d}{dt}p_{k}^{D}(t)=-\mu kp_{k}^{D}(t)+\mu (k+1)p_{k+1}^{D}(t)\quad
\text{for }0\leq k\leq n_{0}-1 \\
\frac{d}{dt}p_{n_{0}}^{D}(t)=-\mu n_{0}p_{n_{0}}^{D}(t)\quad \text{for }%
k=n_{0}%
\end{array}%
\right. ,  \label{acc}
\end{equation}%
with $p_{k}^{D}(0)=1_{\left\{ k=n_{0}\right\} }.$

Moreover, we denote the sublinear death process, with death rate $\mu >0$,
as $\widetilde{D}(t),$ $t\geq 0$ and its distribution is given by%
\begin{equation}
p_{k}^{\widetilde{D}}(t):=\Pr \{\left. \widetilde{D}_{\mu }(t)=k\right\vert
\widetilde{D}_{\mu }(0)=n_{0}\}=\left\{
\begin{array}{l}
e^{-\mu t}(1-e^{-\mu t})^{n_{0}-k},\qquad 1\leq k\leq n_{0} \\
(1-e^{-\mu t})^{n_{0}},\qquad k=0.\label{mo}%
\end{array}%
\right.
\end{equation}%
The probabilities $p_{k}^{\widetilde{D}}$ satisfy the following
difference-differential equations

\begin{equation*}
\left\{
\begin{array}{l}
\frac{d}{dt}p_{k}^{\widetilde{D}}(t)=-\mu (n_{0}-k+1)p_{k}^{\widetilde{D}%
}(t)+\mu (n_{0}-k)p_{k+1}^{\widetilde{D}}(t)\quad 1\leq k\leq n_{0} \\
\frac{d}{dt}p_{0}^{\widetilde{D}}(t)=\mu n_{0}p_{1}^{\widetilde{D}}(t)\quad
k=0%
\end{array}%
\right. ,
\end{equation*}%
with $p_{k}^{\widetilde{D}}(0)=1_{\left\{ k=n_{0}\right\} }.$ In order to
catch the probabilistic mechanism underlying (\ref{mo}) we should consider
that
\begin{equation}
\Pr \{\left. \widetilde{D}_{\mu }(t,t+dt)-\widetilde{D}_{\mu
}(t)=-1\right\vert \widetilde{D}_{\mu }(t)=k\}=\mu (n_{0}-k+1)dt+o(dt),\quad
1\leq k\leq n_{0}.  \notag
\end{equation}%
We note that in the linear death process the probability of cancellation of
an individual in $[t,t+dt)$ is proportional to the number of existing
components at time $t.$ On the other hand, in the sublinear case, this
probability is proportional to the number of deaths recorded up to time $t.$
This property makes the sublinear process more suitable for describing
epidemics and the diffusion of rumors.

The expected value of $\widetilde{D}(t),$ $t\geq 0$ can be evaluated as
follows%
\begin{eqnarray}
\mathbb{E}\widetilde{D}(t) &=&e^{-\mu
t}\sum_{k=0}^{n_{0}}(n_{0}-k)(1-e^{-\mu t})^{k}  \label{e} \\
&=&n_{0}\left[ 1-\left( 1-e^{-\mu t}\right) ^{n_{0}+1}\right] -\frac{1}{\mu }%
(1-e^{-\mu t})\frac{d}{dt}\left[ \sum_{k=0}^{n_{0}}(1-e^{-\mu t})^{k}\right]
\notag \\
&=&n_{0}\left[ 1-\left( 1-e^{-\mu t}\right) ^{n_{0}+1}\right] -(1-e^{-\mu t})%
\left[ e^{\mu t}-e^{\mu t}\left( 1-e^{-\mu t}\right)
^{n_{0}+1}-(n_{0}+1)\left( 1-e^{-\mu t}\right) ^{n_{0}}\right]   \notag \\
&=&n_{0}+\left( 1-e^{-\mu t}\right) ^{n_{0}+1}+e^{\mu t}\left( 1-e^{-\mu
t}\right) ^{n_{0}+2}+1-e^{\mu t}  \notag \\
&=&n_{0}+1-e^{\mu t}[1-(1-e^{-\mu t})^{n_{0}+1}],\qquad n_{0}\geq 1.  \notag
\end{eqnarray}%
It is easy to check that $\lim_{t\rightarrow +\infty }\mathbb{E}\widetilde{D}%
(t)=0,$ as expected. For the linear death process, the hitting time of the $k
$-th event, i.e.%
\begin{equation*}
T_{k}^{D}=\inf \{s>0:D_{\mu }(s)=k\}
\end{equation*}%
has the following distribution%
\begin{equation}
\Pr \{T_{k}^{D}\in dt\}=\binom{n_{0}}{k+1}e^{-\mu t(k+1)}(1-e^{-\mu
t})^{n_{0}-(k+1)}\mu (k+1)dt,\quad t\geq 0,\;0\leq k\leq n_{0}.
\end{equation}%
Thus we get%
\begin{equation*}
\Pr \{T_{k}^{D}<\infty \}=1,
\end{equation*}%
for any $k.$ On the other hand, in the sublinear case, the distribution of
\begin{equation*}
T_{k}^{\widetilde{D}}=\inf \{s>0:\widetilde{D}_{\mu }(s)=k\}
\end{equation*}%
reads%
\begin{equation}
\Pr \{T_{k}^{\widetilde{D}}\in dt\}=e^{-\mu t}(1-e^{-\mu
t})^{n_{0}-(k+1)}\mu (n_{0}-k)dt,\quad t\geq 0,\;0\leq k\leq n_{0}.
\end{equation}%
Moreover we get%
\begin{eqnarray*}
&&\Pr \{T_{k}^{\widetilde{D}}<\infty \}=\int_{0}^{+\infty }\Pr \{T_{k}^{%
\widetilde{D}}\in dt\}dt \\
&=&(n_{0}-k)\sum_{r=0}^{n_{0}-k-1}\binom{n_{0}-k-1}{r}(-1)^{r}\int_{0}^{+%
\infty }e^{-\mu t(r+1)}dt \\
&=&\sum_{r=0}^{n_{0}-k-1}\binom{n_{0}-k}{r+1}(-1)^{r}=1,
\end{eqnarray*}%
for any $k.$

\subsection{Linear subordinated death process}

Let%
\begin{equation*}
\mathcal{Y}(t):=D_{\mu }(N_{\lambda }(t)),\quad t\geq 0,
\end{equation*}%
where $D_{\mu }$ is a linear death process with parameter $\mu $ and $%
N_{\lambda }$ is an independent homogeneous Poisson process with parameter $%
\lambda .$ Then, for $0\leq k\leq n_{0},$ the probability mass function
reads
\begin{eqnarray}
q_{k}^{\mathcal{Y}}(t) &:=&\Pr \{\left. \mathcal{Y}(t)=k\right\vert \mathcal{%
Y}(0)=n_{0}\}  \label{uno} \\
&=&\sum_{l=0}^{\infty }\Pr \{\left. D_{\mu }(l)=k\right\vert D_{\mu
}(0)=n_{0}\}\Pr \{N_{\lambda }(t)=l\}  \notag \\
&=&\binom{n_{0}}{k}\sum_{j=0}^{n_{0}-k}\binom{n_{0}-k}{j}(-1)^{j}\exp
\left\{ -\lambda t(1-e^{-\mu (k+j)})\right\} .  \notag
\end{eqnarray}%
Since%
\begin{equation*}
q_{n_{0}}^{\mathcal{Y}}(t)=e^{-\lambda t(1-e^{-\mu n_{0}})},
\end{equation*}%
the waiting time of the first death of $\mathcal{Y}(t)$ is exponentially
distributed with parameter $\lambda (1-e^{-\mu n_{0}})$ and has the same
form of the corresponding probability of the subordinated linear birth
process, which was given in formula (\ref{st}). The probability generating
function of $\mathcal{Y}(t)$ can be written as follows:%
\begin{eqnarray}
G^{\mathcal{Y}}(u,t) &=&e^{-\lambda t}\sum_{l=0}^{\infty }\frac{(\lambda
t)^{l}}{l!}[1-e^{-\mu l}(1-u)]^{n_{0}}  \label{due} \\
&=&\sum_{m=0}^{n_{0}}\binom{n_{0}}{m}(-1)^{m}(1-u)^{m}\exp \{-\lambda
t(1-e^{-\mu m}\}.  \notag
\end{eqnarray}%
The mean value and variance are respectively given by%
\begin{equation}
\mathbb{E}\mathcal{Y}(t)=n_{0}e^{\lambda t(e^{-\mu }-1)}
\end{equation}%
and
\begin{eqnarray}
Var\mathcal{Y}(t) &=&Var\mathbb{E}\left[ \left. D_{\mu }(N_{\lambda
}(t))\right\vert N_{\lambda }(t)\right] +\mathbb{E}Var\left[ \left. D_{\mu
}(N_{\lambda }(t))\right\vert N_{\lambda }(t)\right] \\
&=&Var\left[ n_{0}e^{-\mu N_{\lambda }(t)}\right] +\mathbb{E}\left[
n_{0}e^{-\mu N_{\lambda }(t)}(1-e^{-\mu N_{\lambda }(t)})\right]  \notag \\
&=&n_{0}(n_{0}-1)e^{\lambda t(e^{-2\mu }-1)}-n_{0}^{2}e^{2\lambda t(e^{-\mu
}-1)}+n_{0}e^{\lambda t(e^{-\mu }-1)}.  \notag
\end{eqnarray}

\

We are interested now in the differential equation satisfied by the
distribution of the process $\mathcal{Y}(t),t>0.$ As a preliminary result we
prove the following lemma.

\begin{lemma}
The probability generating function given in (\ref{due}) satisfies the
following initial-value problem:%
\begin{equation}
\left\{
\begin{array}{l}
\frac{\partial }{\partial t}G^{\mathcal{Y}}(u,t)=-\lambda \left[
1-e^{(1-u)(1-e^{-\mu })\partial _{u}}\right] G^{\mathcal{Y}}(u,t) \\
G^{\mathcal{Y}}(u,0)=u^{n_{0}}%
\end{array}%
\right. ,\quad t,u\geq 0,  \label{tre}
\end{equation}%
where $e^{a\partial _{u}}:=\sum_{k=0}^{\infty }\frac{a^{k}\partial _{u}^{k}}{%
k!}$ is the shift operator.
\end{lemma}

\begin{proof}
By considering the equation governing the state probabilities of the Poisson
process, we can write%
\begin{eqnarray*}
\frac{\partial }{\partial t}G^{\mathcal{Y}}(u,t) &=&\sum_{l=0}^{\infty }%
\frac{\partial }{\partial t}\Pr \{N_{\lambda }(t)=l\}G^{D}(u,l) \\
&=&-\lambda \sum_{l=0}^{\infty }\Pr \{N_{\lambda }(t)=l\}G^{D}(u,l)+\lambda
\sum_{l=1}^{\infty }\Pr \{N_{\lambda }(t)=l-1\}G^{D}(u,l) \\
&=&-\lambda G^{\mathcal{Y}}(u,t)+\lambda \sum_{r=0}^{\infty }\Pr
\{N_{\lambda }(t)=r\}G^{D}(u,r+1) \\
&=&-\lambda G^{\mathcal{Y}}(u,t)+\lambda \sum_{r=0}^{\infty }\Pr
\{N_{\lambda }(t)=r\}e^{\partial _{r}}G^{D}(u,r),
\end{eqnarray*}%
where $G^{D}$ is the probability generating function of $D$. Since, for any $%
t\geq 0,$%
\begin{eqnarray*}
\frac{\partial }{\partial t}G^{D}(u,t) &=&\mu (1-u)\frac{\partial }{\partial
u}G^{D}(u,t) \\
\frac{\partial ^{2}}{\partial t^{2}}G^{D}(u,t) &=&\mu ^{2}(1-u)\frac{%
\partial }{\partial u}\left[ (1-u)\frac{\partial }{\partial u}\right]
G^{D}(u,t) \\
&=&-\mu ^{2}(1-u)\frac{\partial }{\partial u}G^{D}(u,t)+\mu ^{2}(1-u)^{2}%
\frac{\partial ^{2}}{\partial u^{2}}G^{D}(u,t)
\end{eqnarray*}%
\begin{eqnarray*}
\frac{\partial ^{3}}{\partial t^{3}}G^{D}(u,t) &=&\mu ^{3}(1-u)\frac{%
\partial }{\partial u}G^{D}(u,t)-3\mu ^{3}(1-u)^{2}\frac{\partial ^{2}}{%
\partial u^{2}}G^{D}(u,t)+ \\
&&+\mu ^{3}(1-u)^{3}\frac{\partial ^{3}}{\partial u^{3}}G^{D}(u,t)
\end{eqnarray*}%
and, analogously%
\begin{eqnarray*}
\frac{\partial ^{4}}{\partial t^{4}}G^{D}(u,t) &=&-\mu ^{4}(1-u)\frac{%
\partial }{\partial u}G^{D}(u,t)+7\mu ^{4}(1-u)^{2}\frac{\partial ^{2}}{%
\partial u^{2}}G^{D}(u,t)+ \\
&&-6\mu ^{4}(1-u)^{3}\frac{\partial ^{3}}{\partial u^{3}}G^{D}(u,t)+\mu
^{4}(1-u)^{4}\frac{\partial ^{4}}{\partial u^{4}}G^{D}(u,t),
\end{eqnarray*}%
we can write%
\begin{equation*}
\frac{\partial ^{j}}{\partial t^{j}}G^{D}(u,t)=\mu ^{j}\sum_{l=0}^{j-1}{{{{{{%
{{{{{{{{{{{{{\QATOPD \{ \} {j }{j-l}}}}}}}}}}}}}}}}}}}}(-1)^{l}(1-u)^{j-l}%
\frac{\partial ^{j-l}}{\partial u^{j-l}}G^{D}(u,t),
\end{equation*}%
where ${{{{{{{{{{{{{{{{{{{{{{{{{{{{\QATOPD \{ \} {n }{k}}}}}}}}}}}}}}}}}}}}}}%
}}}}}}}$ are the Stirling numbers of the second kind (see \cite{RIO}). Thus
we get (by considering that ${{{{{{{{{{{{{{{{{{{\QATOPD \{ \} {j }{0}}}}}}}}}%
}}}}}}}}}}}=0 $)%
\begin{eqnarray*}
\frac{\partial }{\partial t}G^{\mathcal{Y}}(u,t) &=&-\lambda G^{\mathcal{Y}%
}(u,t)+\lambda \sum_{r=0}^{\infty }\Pr \{N_{\lambda
}(t)=r\}\sum_{j=0}^{\infty }\frac{\mu ^{j}}{j!}\sum_{l=0}^{j}{{{{{{{{{{{{{{{{%
{{{\QATOPD \{ \} {j }{j-l}}}}}}}}}}}}}}}}}}}}(-1)^{l}(1-u)^{j-l}\frac{%
\partial ^{j-l}}{\partial u^{j-l}}G^{D}(u,r) \\
&=&-\lambda G^{\mathcal{Y}}(u,t)+\lambda \sum_{r=0}^{\infty }\Pr
\{N_{\lambda }(t)=r\}\sum_{j=0}^{\infty }\frac{\mu ^{j}}{j!}\sum_{l=0}^{j}{{{%
{{{{{{{{{{{{{{{{\QATOPD \{ \} {j }{l}}}}}}}}}}}}}}}}}}}}(-1)^{j-l}(1-u)^{l}%
\frac{\partial ^{l}}{\partial u^{l}}G^{D}(u,r) \\
&=&-\lambda G^{\mathcal{Y}}(u,t)+\lambda \sum_{j=0}^{\infty }\frac{(-\mu
)^{j}}{j!}B_{j}\left[ (u-1)\frac{\partial }{\partial u}\right] G^{\mathcal{Y}%
}(u,t),
\end{eqnarray*}%
where we have applied the Dobinski formula for the Bell polynomial $B_{n}$,
i.e.%
\begin{equation*}
B_{n}(x)=\sum_{k=0}^{n}x^{k}{{{{{{{{{{{{{{{{{{{\QATOPD \{ \} {n }{k}}}}}}}}}}%
}}}}}}}}}}.
\end{equation*}%
Equation (\ref{tre}) follows by applying the well-known exponential
generating function of the Bell's polynomial (see e.g. \cite{RIO}), i.e.%
\begin{equation*}
\sum_{n=0}^{\infty }\frac{t^{n}B_{n}(x)}{n!}=\exp \{(e^{t}-1)x\}.
\end{equation*}
\end{proof}

\begin{theorem}
The probability mass function of $\mathcal{Y}(t)), t\geq 0$ satisfies the
following equation%
\begin{equation}
\frac{d}{dt} q_{k}^{\mathcal{Y}}(t)=\lambda e^{-\mu k}\sum_{r=0}^{n_{0}-k}%
\binom{r+k}{r} q_{r+k}^{\mathcal{Y}}(t)(1-e^{-\mu })^{r}-\lambda q_{k}^{%
\mathcal{Y}}(t),\qquad 0\leq k\leq n_{0},  \label{bo}
\end{equation}%
with the initial condition
\begin{equation*}
q_{k}^{\mathcal{Y}}(0)=\left\{
\begin{array}{c}
1,\quad k=n_{0} \\
0,\quad k<n_{0}%
\end{array}%
\right. .
\end{equation*}
\end{theorem}

\begin{proof}
In order to derive the differential equation satisfied by (\ref{uno}) we
rewrite the first line in (\ref{tre}) as follows:%
\begin{eqnarray*}
\frac{\partial }{\partial t}G^{\mathcal{Y}}(u,t) &=&\lambda
\sum_{j=0}^{\infty }\frac{(1-u)^{j}(1-e^{-\mu })^{j}}{j!}\partial _{u}^{j}G^{%
\mathcal{Y}}(u,t)-\lambda G^{\mathcal{Y}}(u,t) \\
&=&\lambda \sum_{k=0}^{n_{0}}q_{k}^{\mathcal{Y}}(t)\sum_{j=0}^{\infty }\frac{%
(1-u)^{j}(1-e^{-\mu })^{j}}{j!}\partial _{u}^{j}u^{k}-\lambda G^{\mathcal{Y}%
}(u,t) \\
&=&\lambda \sum_{k=0}^{n_{0}}q_{k}^{\mathcal{Y}}(t)\sum_{j=0}^{\infty }\frac{%
(1-e^{-\mu })^{j}}{j!}k(k-1)...(k-j+1)u^{k-j}\sum_{l=0}^{j}\binom{j}{l}%
(-u)^{l}-\lambda G^{\mathcal{Y}}(u,t) \\
&=&\lambda \sum_{k=0}^{n_{0}}q_{k}^{\mathcal{Y}}(t)\sum_{j=0}^{k}(1-e^{-\mu
})^{j}\binom{k}{j}u^{k-j}\sum_{l=0}^{j}\binom{j}{l}(-u)^{l}-\lambda G^{%
\mathcal{Y}}(u,t) \\
&=&[l^{\prime }=k-j+l] \\
&=&\lambda \sum_{k=0}^{n_{0}}q_{k}^{\mathcal{Y}}(t)\sum_{j=0}^{k}(1-e^{-\mu
})^{j}\binom{k}{j}\sum_{l^{\prime }=k-j}^{k}\binom{j}{k-l^{\prime }}%
(-1)^{j-k+l^{\prime }}u^{l^{\prime }}-\lambda G^{\mathcal{Y}}(u,t) \\
&=&\lambda \sum_{k=0}^{n_{0}}q_{k}^{\mathcal{Y}}(t)\sum_{l=0}^{k}u^{l}\binom{%
k}{l}\sum_{j=k-l}^{k}(1-e^{-\mu })^{j}\binom{l}{k-j}(-1)^{j-k+l}-\lambda G^{%
\mathcal{Y}}(u,t) \\
&=&\lambda \sum_{l=0}^{n_{0}}u^{l}\sum_{k=l}^{n_{0}}\binom{k}{l}q_{k}^{%
\mathcal{Y}}(t)\sum_{j=k-l}^{k}(1-e^{-\mu })^{j}\binom{l}{l-k+j}%
(-1)^{j-k+l}-\lambda G^{\mathcal{Y}}(u,t) \\
&=&[j=k-l+j^{\prime }]
\end{eqnarray*}
\begin{eqnarray*}
&=&\lambda \sum_{l=0}^{n_{0}}u^{l}\sum_{k=l}^{n_{0}}\binom{k}{l}q_{k}^{%
\mathcal{Y}}(t)\sum_{j^{\prime }=0}^{l}(1-e^{-\mu })^{j^{\prime }+k-l}\binom{%
l}{j^{\prime }}(-1)^{j^{\prime }}-\lambda G^{\mathcal{Y}}(u,t) \\
&=&\lambda \sum_{l=0}^{n_{0}}u^{l}\sum_{k=l}^{n_{0}}\binom{k}{l}q_{k}^{%
\mathcal{Y}}(t)(1-e^{-\mu })^{k-l}e^{-\mu l}-\lambda G^{\mathcal{Y}}(u,t) \\
&=&[k=r+l] \\
&=&\lambda \sum_{l=0}^{n_{0}}u^{l}\sum_{r=0}^{n_{0}-l}\binom{r+l}{l}q_{r+l}^{%
\mathcal{Y}}(t)(1-e^{-\mu })^{r}e^{-\mu l}-\lambda G^{\mathcal{Y}}(u,t).
\end{eqnarray*}
\end{proof}

\begin{remark}
By comparing equation (\ref{bo}) with (\ref{acc}), we can see that, in the
subordinated case the time-derivative of $q_{k}^{\mathcal{Y}}(t)$ depends on
the probabilities $q_{r}^{\mathcal{Y}}(t)$, for any $k\leq r\leq n_{0}$,
while, in the standard case, the derivative of $p_{k}^{D}$ depends only on $%
p_{k+1}^{D}.$ This is due to the fact that the process $\mathcal{Y}(t),t>0$
performs downward jumps of arbitrary size, whose law can be derived from (%
\ref{uno}) and can be written as follows:%
\begin{equation}
\Pr \{\left. \mathcal{Y}(t+dt)=k\right\vert \mathcal{Y}(t)=r+k\}=\left\{
\begin{array}{c}
1-\lambda dt+\lambda dte^{-\mu k},\quad r=0 \\
\lambda dt\binom{r+k}{r}(1-e^{-\mu })^{r}e^{-\mu k},\quad r\geq 1.%
\end{array}%
\right.  \label{iss}
\end{equation}%
The first line in (\ref{iss}) means that, during the infinitesimal interval $%
\left[ t,t+dt\right) ,$ either the Poisson process (the "time") does not
change or it changes and all the individuals survive during a time span of
unit length.
\end{remark}

In our view the main result of this section is the density of the first
passage-time through the level\textbf{\ }$k$, i.e.%
\begin{equation*}
T_{k}^{\mathcal{Y}}=\inf \{s>0:\mathcal{Y}(s)=k\}\quad k=0,...n_{0},
\end{equation*}%
which is presented in the following theorem.

\begin{theorem}
The probability density of $T_{k}^{\mathcal{Y}}$ reads, for any $%
k=0,...,n_{0}-1,$%
\begin{eqnarray}
&&\Pr \{\left. T_{k}^{\mathcal{Y}}\in dt\right\vert \mathcal{Y}(0)=n_{0}\}/dt
\label{fr} \\
&=&\lambda e^{-\lambda t-\mu k}\binom{n_{0}}{k}\sum_{j=0}^{\infty }\frac{%
(\lambda t)^{j}e^{-\mu jk}}{j!}\left\{ \left[ 1-e^{-\mu (j+1)}\right]
^{n_{0}-k}-\left[ 1-e^{-\mu j}\right] ^{n_{0}-k}\right\} .  \notag
\end{eqnarray}
\end{theorem}

\begin{proof}
We start by considering that, for $k=0,...,n_{0}-1$%
\begin{eqnarray}
&&\Pr \{\left. T_{k}^{\mathcal{Y}}\in dt\right\vert \mathcal{Y}(0) =n_{0}\}
\label{og} \\
&=&\sum_{h=k+1}^{n_{0}}\Pr \{D_{\mu }(N_{\lambda }(t))=h,D_{\mu }(N_{\lambda
}(t+dt))=k\}  \notag \\
&=&\sum_{h=k+1}^{n_{0}}\Pr \{D_{\mu }(N_{\lambda }(t))=h,D_{\mu }(N_{\lambda
}(t)+dN_{\lambda }(t))=k\}  \notag \\
&=&\lambda dt\sum_{j=1}^{\infty }\sum_{h=k+1}^{n_{0}}\Pr \{D_{\mu
}(j)=h,D_{\mu }(j+1)=k\}\Pr \{N_{\lambda }(t)=j\}+  \notag \\
&&+\lambda dt\Pr \{D_{\mu }(0) = n_{0},D_{\mu }(1)=k\}\Pr \{N_{\lambda
}(t)=0\}  \notag
\end{eqnarray}
\begin{eqnarray}
&=&\lambda dte^{-\lambda t}\sum_{j=1}^{\infty }\frac{(\lambda t)^{j}}{j!}%
\sum_{h=k+1}^{n_{0}}\binom{h}{k}e^{-\mu k}(1-e^{-\mu })^{h-k}\binom{n_{0}}{h}%
e^{-\mu jh}(1-e^{-\mu j})^{n_{0}-h}+  \notag \\
&&+\lambda dte^{-\lambda t}\binom{n_{0}}{k}e^{-\mu k}(1-e^{-\mu })^{n_{0}-k}
\notag \\
&=&\lambda dte^{-\lambda t-\mu k}\binom{n_{0}}{k}\sum_{j=1}^{\infty }\frac{%
(\lambda t)^{j}}{j!}\sum_{h=k+1}^{n_{0}}\binom{n_{0}-k}{n_{0}-h}(1-e^{-\mu
})^{h-k}e^{-\mu jh}(1-e^{-\mu j})^{n_{0}-h}+  \notag \\
&&+\lambda dte^{-\lambda t}\binom{n_{0}}{k}e^{-\mu k}(1-e^{-\mu })^{n_{0}-k}
\notag \\
&=&\lambda dte^{-\lambda t-\mu k}\binom{n_{0}}{k}\sum_{j=1}^{\infty }\frac{%
(\lambda t)^{j}}{j!}e^{-\mu jk}\left\{ \left[ e^{-\mu j}(1-e^{-\mu
})+1-e^{-\mu j}\right] ^{n_{0}-k}-\left[ 1-e^{-\mu j}\right]
^{n_{0}-k}\right\} +  \notag \\
&&+\lambda dte^{-\lambda t}\binom{n_{0}}{k}e^{-\mu k}(1-e^{-\mu })^{n_{0}-k},
\notag
\end{eqnarray}
which coincides with (\ref{fr}).
\end{proof}

\begin{remark}
By integrating (\ref{fr}) we get that, for $k=0,...,n_{0}-1$%
\begin{equation}
\Pr \{\left. T_{k}^{\mathcal{Y}}<\infty \right\vert \mathcal{Y}%
(0)=n_{0}\}=e^{-\mu k}\binom{n_{0}}{k}\sum_{j=0}^{\infty }e^{-\mu jk}\left\{ %
\left[ 1-e^{-\mu (j+1)}\right] ^{n_{0}-k}-\left[ 1-e^{-\mu j}\right]
^{n_{0}-k}\right\} .  \label{ic3}
\end{equation}%
We note that the probability (\ref{ic3}) has the same structure of formula
(36) in \cite{OP}, which is related to the iterated Poisson process, despite
the fact that the outer process $D$ has decreasing paths.
\end{remark}

\begin{remark}
An alternative form of (\ref{fr}), as a finite sum, can be obtained as
follows
\begin{equation*}
\Pr \{\left. T_{k}^{\mathcal{Y}}\in dt\right\vert \mathcal{Y}%
(0)=n_{0}\}/dt=\lambda e^{-\mu k}\binom{n_{0}}{k}\sum_{m=1}^{n_{0}-k}\binom{%
n_{0}-k}{m}(-1)^{m-1}(1-e^{-\mu m})\exp [-\lambda t(1-e^{\mu (k+m)})],
\end{equation*}%
which, by integration, gives%
\begin{equation*}
\Pr \{\left. T_{k}^{\mathcal{Y}}<\infty \right\vert \mathcal{Y}%
(0)=n_{0}\}=e^{-\mu k}\binom{n_{0}}{k}\sum_{m=0}^{n_{0}-k}(-1)^{m-1}\binom{%
n_{0}-k}{m}\frac{1-e^{-\mu m}}{1-e^{-\mu (m+k)}},
\end{equation*}%
for any $k\geq 0.$We note that, for $k=0,$ the extinction probability is
given by%
\begin{equation*}
\Pr \{\left. T_{0}^{\mathcal{Y}}<\infty \right\vert \mathcal{Y}%
(0)=n_{0}\}=-\sum_{r=1}^{n_{0}}(-1)^{r}\binom{n_{0}}{r}=-%
\sum_{r=0}^{n_{0}}(-1)^{r}\binom{n_{0}}{r}+1=1.
\end{equation*}%
For $k=n_{0}-1$, we have instead that%
\begin{eqnarray*}
\Pr \{\left. T_{n_{0}-1}^{\mathcal{Y}}<\infty \right\vert \mathcal{Y}%
(0)=n_{0}\} &=&e^{-\mu (n_{0}-1)}n_{0}\frac{1-e^{-\mu }}{1-e^{-\mu n_{0}}} \\
&=&e^{-\mu (n_{0}-1)}\frac{n_{0}}{1+e^{-\mu }+e^{-2\mu }+...+e^{-\mu
(n_{0}-1)}} \\
&<&e^{-\mu (n_{0}-1)}\frac{n_{0}}{n_{0}e^{-\mu (n_{0}-1)}}=1.
\end{eqnarray*}
\end{remark}

\begin{figure}[htp!]
\caption{Hitting times $\Pr \{T_{k}^{\mathcal{Y}}<\infty \}$}
\label{label2}\centering
\par
\includegraphics[scale=0.30]{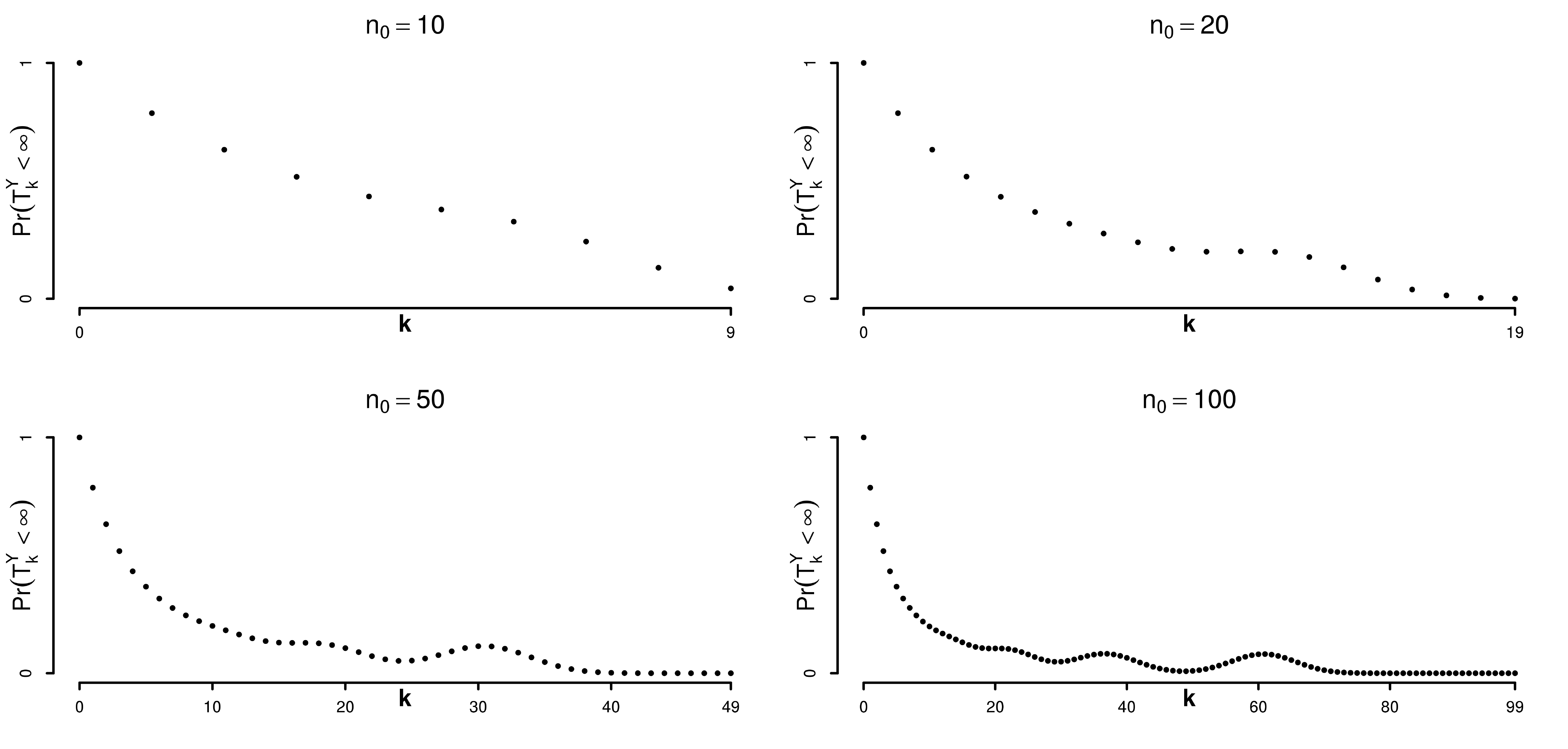}
\end{figure}

\

Unexpectedly enough Fig.3 (which is obtained here for $\mu =1/2$) shows that
the probabilities $\Pr \{\left. T_{k}^{\mathcal{Y}}<\infty \right\vert
\mathcal{Y}(0)=n_{0}\}$ do not display a monotonic behavior for sufficiently
large values of $n_{0}$.

\subsection{Sublinear subordinated death process}

Let
\begin{equation*}
\widetilde{\mathcal{Y}}(t):=\widetilde{D}_{\mu }(N_{\lambda }(t))
\end{equation*}%
where $\widetilde{D}_{\mu }$ is a sublinear death process with parameter $%
\mu $ and $N_{\lambda }$ is an independent Poisson process with parameter $%
\lambda .$ Then the probability mass function reads
\begin{eqnarray}
q_{k,n_{0}}^{\widetilde{\mathcal{Y}}}(t)&:=&\Pr \{\left. \widetilde{%
\mathcal{Y}}(t)=k\right\vert \widetilde{\mathcal{Y}}(0)=n_{0}\}  \label{no}
\\
&=&\sum_{l=0}^{\infty }\Pr \{\left. \widetilde{D}_{\mu }(l)=k\right\vert
\widetilde{D}_{\mu }(0)=n_{0}\}\Pr \{N_{\lambda }(t)=l\}  \notag \\
&=&e^{-\lambda t}\sum_{j=0}^{n_{0}-k}\binom{n_{0}-k}{j}(-1)^{j}\exp \left\{
\lambda te^{-\mu (1+j)}\right\} ,  \notag
\end{eqnarray}%
for $1\leq k\leq n_{0}$, while, for $k=0$, is given by%
\begin{equation}
q_{0,n_{0}}^{\widetilde{\mathcal{Y}}}(t)=e^{-\lambda t}\sum_{j=0}^{n_{0}}%
\binom{n_{0}}{j}(-1)^{j}\exp \left\{ \lambda te^{-\mu j}\right\} .
\label{no2}
\end{equation}%
Clearly $q_{0,n_{0}}^{\widetilde{\mathcal{Y}}}$ is a decreasing function of
the initial number of individuals, as can be directly checked, by
considering (\ref{mo}). The expected value of $\widetilde{\mathcal{Y}}(t),$ $%
t\geq 0$ can be evaluated, by considering (\ref{e}), as follows%
\begin{eqnarray*}
\mathbb{E}\widetilde{\mathcal{Y}}(t) &=&e^{-\lambda t}\sum_{j=0}^{\infty
}\left\{ n_{0}+1-e^{\mu j}[1-(1-e^{-\mu j})^{n_{0}+1}]\right\} \frac{%
(\lambda t)^{j}}{j!} \\
&=&n_{0}+1-e^{-\lambda t(1-e^{\mu })}+\sum_{j=0}^{\infty }(1-e^{-\mu
j})^{n_{0}+1}\frac{(\lambda te^{\mu })^{j}}{j!} \\
&\leq &n_{0}-\left( 1-e^{-\lambda t(1-e^{-\mu })}\right) .
\end{eqnarray*}%
In the case of the subordinated sublinear death process, we have no
markovianity and thus we cannot evaluate explicitly the distribution of the
hitting times $T_{k}^{\widetilde{\mathcal{Y}}},$ $0\leq k\leq n_{0}-1.$ For
this reason we consider here the instant of the first downcrossing of the
level $k$, i.e.%
\begin{equation*}
V_{k}^{\widetilde{\mathcal{Y}}}=\inf \{s>0:\widetilde{\mathcal{Y}}(s)\leq
k\},\qquad 0\leq k\leq n_{0}-1.
\end{equation*}%
For $k\geq 1$, we can write
\begin{eqnarray*}
\Pr \{V_{k}^{\widetilde{\mathcal{Y}}}>t\} &=&\Pr \{\widetilde{\mathcal{Y}}%
(t)\geq k\} \\
&=&[\text{by (\ref{no})}] \\
&=&\sum_{l=k}^{n_{0}}\sum_{r=0}^{n_{0}-l}\binom{n_{0}-l}{r}(-1)^{r}\exp
\left\{ -\lambda t(1-e^{-\mu (1+r)})\right\}  \\
&=&\sum_{r=0}^{n_{0}-k}(-1)^{r}\exp \left\{ -\lambda t(1-e^{-\mu
(1+r)})\right\} \sum_{l=k}^{n_{0}-r}\binom{n_{0}-l}{r}.
\end{eqnarray*}%
The inner sum can be treated as follows%
\begin{eqnarray*}
\sum_{l=k}^{n_{0}-r}\binom{n_{0}-l}{r} &=&\sum_{l=0}^{n_{0}-r-k}\binom{%
n_{0}-l-k}{r} \\
&=&[\text{by eq. (2.25) in \cite{BOS}}] \\
&=&\binom{n_{0}-k+1}{r+1},
\end{eqnarray*}%
so that we get%
\begin{equation*}
\Pr \{V_{k}^{\widetilde{\mathcal{Y}}}>t\}=\sum_{r=1}^{n_{0}-k+1}(-1)^{r-1}%
\binom{n_{0}-k+1}{r}\exp \left\{ -\lambda t(1-e^{-\mu r})\right\} ,
\end{equation*}%
which vanishes, for $t\rightarrow \infty $, for any $k\geq 1$. For $k=0,$ by
applying (\ref{no2}) we have, instead, the following extinction probability:%
\begin{equation*}
\Pr \{V_{0}^{\widetilde{\mathcal{Y}}}<t\}=\Pr \{\widetilde{\mathcal{Y}}%
(t)=0\}=e^{-\lambda t}\sum_{j=0}^{n_{0}}\binom{n_{0}}{j}(-1)^{j}\exp \left\{
\lambda te^{-\mu j}\right\} ,
\end{equation*}%
which, for $t\rightarrow \infty $, is equal to $1.$

\section*{Acknowledgement}
We thank Dr. Bruno Toaldo for providing the figures presented in this paper.

\end{document}